\documentclass[11 pt]{amsart}
\usepackage{amsmath, amsthm,amssymb,bbm,enumerate,mathrsfs,mathtools}
\usepackage{a4wide}
\usepackage{tikz,xcolor,verbatim}
\usetikzlibrary{calc,shapes}
\usepackage{hyperref}
\usepackage{makecell}
\usepackage[numbers,sort&compress]{natbib}

\usepackage[yyyymmdd,hhmmss]{datetime}
 \usepackage{background}
 \SetBgContents{ver. \currenttime \qquad \today \qquad
       This is a preprint version of a submitted manuscript. }
 \SetBgScale{1}
 \SetBgAngle{0}
 \SetBgOpacity{1}
 \SetBgPosition{current page.south west}
 \SetBgHshift{8cm}
 \SetBgVshift{1cm}

\tikzset{
  bigblue/.style={circle, draw=blue!80,fill=blue!40,thick, inner sep=1.5pt, minimum size=5mm},
  bigred/.style={circle, draw=red!80,fill=red!40,thick, inner sep=1.5pt, minimum size=5mm},
  bigblack/.style={circle, draw=black!100,fill=black!40,thick, inner sep=1.5pt, minimum size=5mm},
  bluevertex/.style={circle, draw=blue!100,fill=blue!100,thick, inner sep=0pt, minimum size=2mm},
  redvertex/.style={circle, draw=red!100,fill=red!100,thick, inner sep=0pt, minimum size=2mm},
  blackvertex/.style={circle, draw=black!100,fill=black!100,thick, inner sep=0pt, minimum size=2mm},  
  whitevertex/.style={circle, draw=black!100,fill=white!100,thick, inner sep=0pt, minimum size=2mm},  
  smallblack/.style={circle, draw=black!100,fill=black!100,thick, inner sep=0pt, minimum size=1mm},
    microvert/.style={circle, draw=black!100,fill=black!100, inner sep=0pt, minimum size=0.0001mm},
  smallwhite/.style={circle, draw=black!100,fill=white!100,thick, inner sep=0pt, minimum size=1mm} 
}

\makeatletter
\pgfdeclareshape{myNode}{
  \inheritsavedanchors[from=rectangle] 
  \inheritanchorborder[from=rectangle]
  \inheritanchor[from=rectangle]{center}
  \inheritanchor[from=rectangle]{north}
  \inheritanchor[from=rectangle]{south}
  \inheritanchor[from=rectangle]{west}
  \inheritanchor[from=rectangle]{east}
  \backgroundpath{
    \southwest \pgf@xa=\pgf@x \pgf@ya=\pgf@y
    \northeast \pgf@xb=\pgf@x \pgf@yb=\pgf@y
    \pgfsetcornersarced{\pgfpoint{5pt}{5pt}}
    \pgfpathmoveto{\pgfpoint{\pgf@xa}{\pgf@ya}}
    \pgfpathlineto{\pgfpoint{\pgf@xa}{\pgf@yb}}
    \pgfpathlineto{\pgfpoint{\pgf@xb}{\pgf@yb}}
    \pgfsetcornersarced{\pgfpoint{5pt}{5pt}}
    \pgfpathlineto{\pgfpoint{\pgf@xb}{\pgf@ya}}
    \pgfpathclose
 }
}
\makeatother

\title{Graph Homomorphism Reconfiguration and Frozen $\boldsymbol{H}$-Colourings}
\author[Brewster]{Richard C. Brewster}
\address[Richard C. Brewster]{Department of Mathematics and Statistics, Thompson Rivers University, Kamloops, BC, Canada} 
\email{rbrewster@tru.ca}

\author[Lee]{Jae-Baek Lee}
\address[Jae-Baek Lee and Mark Siggers]{College of Natural Sciences, Kyungpook National University, Daegu 702-701, South Korea}
\email{dlwoqor0923@gmail.com, mhsiggers@knu.ac.kr}

\author[Moore]{Benjamin Moore}
\address[Benjamin Moore]{Department of Combinatorics and Optimization, University of Waterloo, Waterloo, ON, Canada}  
\email{brmoore@uwaterloo.ca}

\author[Noel]{Jonathan A. Noel}
\address[Jonathan A. Noel]{Department of Computer Science and DIMAP, University of Warwick, Coventry CV4 7AL.}  
\email{j.noel@warwick.ac.uk}
              
\thanks{
The First author is supported by the Natural Sciences and Engineering Research Council of Canada Grant RGPIN-2014-04760.
The second author is supported by the Kyungpook University BK21 Grant.
Part of this work was done while the third author was a student at Simon Fraser University and supported by NSERC. 
Part of this work was done while the fourth author was a DPhil student at the University of Oxford and part of this work was done while the fourth author was a postdoctoral researcher at ETH Z\"{u}rich.
The fifth author is supported by Korean NRF Basic Science Research Program (2015-R1D1A1A01057653) funded by the Korean government (MEST)
                  and the Kyungpook National University Research Fund. The second and fifth authors would like to thank the first author and Thompson Rivers Univesity for their support during the writing of this work.
}

\author[Siggers]{Mark Siggers}

\date{}

\newtheorem{thm}[equation]{Theorem}
\newtheorem{lem}[equation]{Lemma}
\newtheorem{prop}[equation]{Proposition}

\newtheorem{cor}[equation]{Corollary}
\newtheorem{claim}[equation]{Claim}
\newtheorem{ques}[equation]{Question}

\theoremstyle{definition}
\newtheorem{defn}[equation]{Definition}
\newtheorem{obs}[equation]{Observation}

\newtheorem*{ack}{Acknowledgements}

\newtheorem{rem}[equation]{Remark}

\newtheoremstyle{case}{}{}{\normalfont}{}{\itshape}{\normalfont:}{ }{}

\theoremstyle{case}

\newtheorem{step}{Step}

\numberwithin{equation}{section}

\newcommand\Hfrz[1]{\textsc{Frozen $#1$-Colouring}}

\newcommand\CSP[1]{\textsc{CSP($#1$)}}
\newcommand\Hcol[1]{\textsc{$#1$-Colouring}}
\newcommand\Hrec[1]{\textsc{$#1$-Recolouring}}

\newcommand{\Hom}{\operatorname{Hom}}

\newcommand{\bHom}{\operatorname{\mathbf{Hom}}}

\newcommand{\kgraph}{$k$-relation}
\newcommand{\kgraphs}{$k$-relations}

\begin{document}

\begin{abstract}
For a fixed graph $H$, the \emph{reconfiguration problem} for $H$-colourings (i.e. homomorphisms to $H$) asks: given a graph $G$ and two $H$-colourings $\varphi$ and $\psi$ of $G$, does there exist a sequence $f_0,\dots,f_m$ of $H$-colourings such that $f_0=\varphi$, $f_m=\psi$ and $f_i(u)f_{i+1}(v)\in E(H)$ for every $0\leq i<m$ and $uv\in E(G)$? If the graph $G$ is loop-free, then this is the equivalent to asking whether it possible to transform $\varphi$ into $\psi$ by changing the colour of one vertex at a time such that all intermediate mappings are $H$-colourings. In the affirmative, we say that $\varphi$ \emph{reconfigures} to $\psi$. Currently, the complexity of deciding whether an $H$-colouring $\varphi$ reconfigures to an $H$-colouring $\psi$ is only known when $H$ is a clique, a circular clique, a $C_4$-free graph, or in a few other cases which are easily derived from these. We show that this problem is PSPACE-complete when $H$ is an odd wheel. 

An important notion in the study of reconfiguration problems for $H$-colourings is that of a \emph{frozen $H$-colouring}; i.e. an $H$-colouring $\varphi$ such that $\varphi$ does not reconfigure to any $H$-colouring $\psi$ such that $\psi\neq \varphi$. We obtain an explicit dichotomy theorem for the problem of deciding whether a given graph $G$ admits a frozen $H$-colouring. The hardness proof involves a reduction from a CSP problem which is shown to be NP-complete by establishing the non-existence of a certain type of polymorphism. 
\end{abstract}

\maketitle

\section{Introduction}

All graphs in this paper are finite and undirected (although, some of the problems that we consider are also interesting for directed graphs; see~\cite{directed}). We allow loops but no multiple edges; a vertex is said to be \emph{reflexive} if it has a loop.  A \emph{homomorphism} from a graph $G$ to a graph $H$, sometimes called an \emph{$H$-colouring} of $G$, is a mapping $\varphi:V(G)\to V(H)$ such that $\varphi(u)\varphi(v)\in E(H)$ whenever $uv\in E(G)$. We write $G\to H$ to indicate that $G$ admits a homomorphism to $H$ and $\varphi:G\to H$ to refer to a particular homomorphism $\varphi$. The vertices of $H$ are often referred to as \emph{colours} and, for $v\in V(G)$, the image $f(v)$ of $v$ is often referred to as the \emph{colour} of $v$. Given graphs $G$ and $H$, let  $\Hom(G,H)$ denote the set of all homomorphisms from $G$ to $H$. 

The most fundamental decision problem for graph homomorphisms is the \Hcol{H} problem, which asks whether a given graph $G$ admits a homomorphism to $H$. This problem is easily solvable in polynomial time if $H$ is bipartite or contains a loop. Hell and Ne\v{s}et\v{r}il~\cite{Hcol} famously proved that, in all other cases, the problem is NP-complete. More recently, new proofs of this result using techniques from universal algebra~\cite{Bulatov} and Fourier analysis~\cite{KunSzegedy} have been discovered, as well as a shorter purely combinatorial proof~\cite{Siggers}.

Our focus in this paper is on two decision problems for graph homomorphisms arising from an area of research known as ``combinatorial reconfiguration.'' A standard reconfiguration problem asks, given two solutions to a fixed combinatorial problem, whether it is possible to transform one of the solutions into the other by applying a sequence of allowed modifications. Results in combinatorial reconfiguration can provide interesting insights into the structure of the solution space of a combinatorial problem and, sometimes, ideas from combinatorial reconfiguration can even be used to establish the existence of a solution with special properties. For example, Wrochna~\cite{WrochnaMult} used ideas from his own paper on combinatorial reconfiguration~\cite{Wrochna} to prove that every graph $H$ without cycles of length $4$ is ``multiplicative'' in the sense that $G\times F\to H$ implies that $G\to H$ or $F\to H$ (see Definition~\ref{catProdDef} for a definition of the $\times$ product for graphs). For more background on combinatorial reconfiguration, one should consult the broad surveys of van den Heuvel~\cite{changeSurvey} and Nishimura~\cite{Nishimura}. 

Given $\varphi,\psi\in \Hom(G,H)$, a \emph{reconfiguration sequence} taking $\varphi$ to $\psi$ is a sequence $f_0,\dots,f_m\in \Hom(G,H)$ such that $f_0=\varphi$, $f_m=\psi$ and, for $0\leq i< m$, we have $f_i(u)f_{i+1}(v)\in E(H)$ for all $uv\in E(G)$. If there exists a reconfiguration sequence taking $\varphi$ to $\psi$, then we say that $\varphi$ \emph{reconfigures} to $\psi$. Given a fixed graph $H$, the following decision problem is known as \Hrec{H}:
\begin{itemize}
\item[] \textbf{Instance:} A graph $G$ and $\varphi,\psi\in \Hom(G,H)$.
\item[] \textbf{Question:} Does $\varphi$ reconfigure to $\psi$?
\end{itemize}
One may observe that, if $\varphi$ reconfigures to $\psi$, then there exists a reconfiguration sequence  taking $\varphi$ to $\psi$ in which any two consecutive elements of the sequence differ in only one vertex. From this, it is not hard to see that, if $G$ is loop-free and $\varphi,\psi\in\Hom(G,H)$, then $\varphi$ reconfigures to $\psi$ if and only if $\varphi$ can be transformed into $\psi$ by changing the colour of one vertex at a time while maintaining that the mapping is an $H$-colouring. For general $G$, one requires the extra condition that, if $u\in V(G)$ is reflexive, then the colour of $u$ must always be changed to a neighbour of its current colour. 

Problems in combinatorial reconfiguration can usually be viewed as questions about the structure of a so called ``reconfiguration graph:''  the vertices of this graph are solutions and the edges correspond to a single application of a reconfiguration step.  In our case, the reconfiguration graph, denoted $\bHom(G,H)$, has $\Hom(G,H)$ as its vertex set where two homomorphisms $f$ and $g$ are adjacent whenever $f(u)g(v)\in E(H)$ for all $uv\in E(G)$.  Thus \Hrec{H} asks whether $\varphi,\psi\in \Hom(G,H)$ are in the same component of $\bHom(G,H)$~\cite{BrewsterNoel,3colReconfig}.

It is trivial that \Hrec{H} is solvable in polynomial time if $H=K_1$ or $H=K_2$, and it is also not hard to see that \Hrec{H} is contained in PSPACE for general fixed $H$. Cereceda, van den Heuvel and Johnson~\cite{3colReconfig} proved that, surprisingly, \Hrec{K_3} is solvable in polynomial time despite the fact that the \Hcol{K_3} problem is NP-complete (i.e. it is NP-complete to decide if a graph admits a proper $3$-colouring). On the other hand, Bonsma and Cereceda~\cite{Bonsma} proved that \Hrec{K_k} is PSPACE-complete for all $k\geq4$. They achieved this by showing that the decision problem \textsc{Sliding Tokens}, which was shown to be PSPACE-complete in~\cite{Sliding}, reduces to \Hrec{K_k}. Thus, the reconfiguration problem for homomorphims to cliques admits the following dichotomy theorem. 

\begin{thm}[Cereceda, van den Heuvel and Johnson~\cite{3colReconfig}; Bonsma and Cereceda~\cite{Bonsma}]
\label{colouringDichotomy}
The \Hrec{K_k} problem is solvable in polynomial time if $k\leq 3$ and is PSPACE-complete if $k\geq4$. 
\end{thm}

There has been some recent work on extending Theorem~\ref{colouringDichotomy} to larger classes of graphs. Wrochna~\cite{Wrochna} extended the polynomial side of Theorem~\ref{colouringDichotomy} to the following remarkably general result: \emph{If $H$ does not contain a $4$-cycle, then \Hrec{H} is solvable in polynomial time}. Brewster, McGuinness, Moore and Noel~\cite{circularReconfig} generalised both the polynomial and PSPACE-complete sides of Theorem~\ref{colouringDichotomy} to circular $(p,q)$-cliques (see~\cite{circularReconfig} for a definition): the problem is polynomial if $p/q < 4$ and PSPACE-complete otherwise. Intriguingly, the class of graphs for which \Hrec{H} is known to be solvable in polynomial time coincides exactly with the class of graphs which are known to be multiplicative. The fact that circular $(p,q)$-cliques with $p/q<4$ are multiplicative was first proved by Tardif~\cite{Tardif2005}; a new proof, using ideas from reconfiguration, was recently given by Wrochna~\cite{WrochnaMult}. More recently, Brewster, Lee and Siggers~\cite{directed} have investigated the complexity of \Hrec{H} for digraphs in which every vertex is reflexive (where digraph homomorphisms are the same as graph homomorphisms, except that they are required to preserve the directions of the arcs). 

Our first theorem extends the list of ``hardness'' results for \Hrec{H} to a new family of graphs. The \emph{wheel} of length $m$, denoted $W_m$, is the graph obtained from a cycle of length $m$  by adding a vertex adjacent to all vertices of the cycle. We prove the following.

\begin{thm}
\label{wheelThm}
For $k\geq1$, \Hrec{W_{2k+1}} is PSPACE-complete.
\end{thm}

An important first step in each of the polynomial-time algorithms for \Hrec{H} in~\cite{3colReconfig,Wrochna,circularReconfig} is to determine the set of vertices $v$ of $G$ such that $\varphi'(v)=\varphi(v)$ for every homomorphism $\varphi'$ which reconfigures to $\varphi$; i.e. to determine the set of vertices which cannot change their colour under any reconfiguration sequence starting with $\varphi$. We say that such a vertex is \emph{frozen} by $\varphi$ and that the homomorphism $\varphi$ itself is \emph{frozen} if every vertex of $G$ is frozen by $\varphi$. Clearly, a frozen homomorphism corresponds to an isolated vertex in $\bHom(G,H)$. We consider the following decision problem, which we call \Hfrz{H}:
\begin{itemize}
\item[] \textbf{Instance:} A graph $G$.
\item[] \textbf{Question:} Does there exist a frozen $H$-colouring of $G$?
\end{itemize}
Clearly, \Hfrz{H} is in NP. As it turns out, there are many graphs $H$ for which there does not even exist a graph $G$ admitting a frozen $H$-colouring; as a simple example, consider a complete multipartite graph in which every part has size at least two. We say that a connected graph $H$ is \emph{thermal} if there does not exist a graph $G$ such that $V(G)\neq \emptyset$ and $G$ has a frozen $H$-colouring. If $H$ is thermal, then \Hfrz{H} is trivially solvable in polynomial time (as the answer is always ``no''). We prove a dichotomy theorem for \Hfrz{H} which is stated in full generality in Section~\ref{frozenSection}. Here, we state the theorem only for connected graphs.

\begin{thm}
\label{frozenThmConn}
Let $H$ be a connected graph. If $H$ has at most two vertices or $H$ is thermal, then \Hfrz{H} is solvable in polynomial time. Otherwise, \Hfrz{H} is NP-complete. 
\end{thm}

In the next section, we prove Theorem~\ref{wheelThm} by showing that, for $k\geq2$, the \Hrec{K_{2k+1}} problem can be reduced to the \Hrec{W_{2k+1}} problem, and therefore the latter is PSPACE-complete by Theorem~\ref{colouringDichotomy}. This proof involves an intermediate reduction to and from a reconfiguration problem for so called ``edge-coloured homomorphisms.'' In Section~\ref{frozenSection}, we prove a dichotomy theorem for \Hfrz{H}. The hardness proof involves showing that a certain constraint satisfaction problem, which we show is NP-complete using algebraic techniques,  reduces to the \Hfrz{H} problem. We conclude the paper in Section~\ref{conclusion} with some remarks and open problems. 

\section{Odd Wheel Recolouring is Hard}
\label{wheelsSec}

\subsection{Preliminaries}
Let us first take the time to fix some notation and terminology and make some preliminary observations. Note that $W_3$ is simply $K_4$ and so  \Hrec{W_3} is PSPACE-complete by Theorem~\ref{colouringDichotomy}. Thus, it suffices to prove Theorem~\ref{wheelThm} for $k\geq2$. The following two definitions are standard. 

\begin{defn}
Let $C_m$ denote the cycle with vertex set $\{0,\dots,m-1\}$ where $ij$ is an edge if $i\equiv j\pm1\bmod m$.
\end{defn}

\begin{defn}
Let $P_m$ be the path on $m$ vertices obtained from $C_m$ by deleting the edge from $0$ to $m-1$. 
\end{defn}

We write the vertex set of $W_{2k+1}$ as $\{0,\dots, 2k\}\cup\{\alpha\}$ where the vertices of $\{0,\dots,2k\}$ form a cycle on $2k+1$ vertices with adjacencies are as in $C_{2k+1}$ and $\alpha$ is adjacent to every vertex of $\{0,\dots,2k\}$. We always view the colours in $\{0,\dots,2k\}$ modulo $2k+1$.

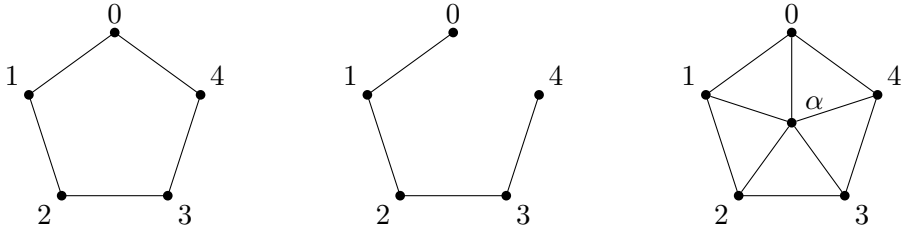
\begin{figure}[htbp]

\begin{tikzpicture}
   \newdimen\R
   \R=1.2cm
  \draw (18:\R)
\foreach [count=\n from 0] \x in {90,162,...,378} {
	-- (\x:\R) node [smallblack, label={[label distance=-2pt]\x:{\n}}] {}
};

\begin{scope}[shift={(4.5,0)}]

  \draw (90:\R)
\foreach [count=\n from 0] \x in {90,162,...,378} {
	-- (\x:\R) node [smallblack, label={[label distance=-2pt]\x:{\n}}] {} 
};
\end{scope}

\begin{scope}[shift={(9,0)}]
  \draw (18:\R)
\foreach [count=\n from 0] \x in {90,162,...,378} {
	-- (\x:\R) node [smallblack, label={[label distance=-2pt]\x:{\n}}] (v\n){}
};
\draw (0,0) node [smallblack, label={[label distance=0pt]45:{$\alpha$}}] (alpha) {};
\draw (alpha) -- (v0);
\draw (alpha) -- (v1);
\draw (alpha) -- (v2);
\draw (alpha) -- (v3);
\draw (alpha) -- (v4);
\end{scope}

\end{tikzpicture}

\caption{The graphs $C_5$, $P_5$ and $W_5$.}
\label{C5P5W5Fig}
\end{figure}

In the proof of Theorem~\ref{wheelThm}, it will be convenient to reduce \Hrec{K_{2k+1}} to an intermediate reconfiguration problem for so called ``edge-coloured homomorphisms,'' defined below, and then to reduce this problem to \Hrec{W_{2k+1}}. 

\begin{defn}
An \emph{edge-coloured graph} is a tuple $G=(G_1,\dots,G_k)$ for some $k\geq1$, where $G_1,\dots, G_k$ are graphs with the same vertex set. The \emph{vertex set} of $G$ is taken to be $V(G):=V(G_1)$. 
\end{defn}

\begin{defn}
Given edge-coloured graphs $G=(G_1,\dots, G_k)$ and $H=(H_1,\dots, H_k)$, a \emph{homomorphism} from $G$ to $H$, or an \emph{$H$-colouring} of $G$, is a function $f:V(G)\to V(H)$ such that $f$ is a homomorphism from $G_i$ to $H_i$ for all $1\leq i\leq k$. 
\end{defn}

One should think of an edge-coloured graph $G=(G_1,\dots,G_k)$ as a multigraph obtained from  superimposing the graphs $G_1,\dots, G_k$ on top of each other (keeping any multiple edges that arise) and colouring the edges of $G_i$ with colour $i$ for $1\leq i\leq k$. In this sense, an edge-coloured homomorphism is a mapping which preserves ``coloured adjacencies.'' Problems regarding edge-coloured homomorphisms are well studied; see, e.g.,~\cite{RickThesis, signed17, Alon, BrewRel}. 

Given an edge-coloured graph $H=(H_1, H_2, \dots, H_k)$, the \Hrec{H} problem extends the definition for graphs in the natural way: the input is an edge-coloured graph $G$ and a pair of $H$-colourings $\varphi$ and $\psi$ of $G$ and the goal is to decide whether $\varphi$ there exists $H$-colourings $f_0,\dots,f_m$ such that $f_0=\varphi$, $f_m=\psi$ and $f_i(u)f_{i+1}(v)\in E(H_j)$ for every $0\leq i<m$, $1\leq j\leq k$ and $uv\in E(G_j)$. The proof of Theorem~\ref{wheelThm} is divided into the following two lemmas; first, we give a definition.

\begin{defn}
For $m\geq1$, let $Z_m$ be the graph obtained from a clique on vertex set $\{0,\dots,m\}\cup \{\alpha\}$ by adding a loop at every vertex except for $\alpha$. 
\end{defn}

\begin{figure}[htbp]

\vspace{-0.5cm}

\begin{tikzpicture}
   \newdimen\R
   \R=1.2cm
  \draw (18:\R)
\foreach [count=\n from 0] \x in {90,162,...,378} {
	-- (\x:\R) node [smallblack, label={[label distance=8pt]\x:{\n}}] (v\n){}
};

  \draw (18:1.3*\R)
\foreach [count=\n from 0] \x in {90,162,...,378} {
	 (\x:1.3*\R) node [microvert] (loop\n){}
};

\draw (0,0) node [smallblack, label={[label distance=0pt]45:{$\alpha$}}] (alpha) {};
\draw (alpha) -- (v0);
\draw (alpha) -- (v1);
\draw (alpha) -- (v2);
\draw (alpha) -- (v3);
\draw (alpha) -- (v4);

\draw [black] (v1) to[out=18,in=162]  (v4);
\draw [black] (v0) to[out=234,in=90]  (v2);
\draw [black] (v1) to[out=306,in=162]  (v3);
\draw [black] (v2) to[out=378,in=234]  (v4);
\draw [black] (v3) to[out=450,in=306]  (v0);

\draw [black] (v0) to[out=45,in=0]  (loop0);
\draw [black] (loop0) to[out=180,in=135]  (v0);

\draw [black] (v1) to[out=117,in=72]  (loop1);
\draw [black] (loop1) to[out=252,in=207]  (v1);

\draw [black] (v2) to[out=189,in=144]  (loop2);
\draw [black] (loop2) to[out=324,in=279]  (v2);

\draw [black] (v3) to[out=271,in=216]  (loop3);
\draw [black] (loop3) to[out=396,in=351]  (v3);

\draw [black] (v4) to[out=343,in=288]  (loop4);
\draw [black] (loop4) to[out=468,in=423]  (v4);

\end{tikzpicture}

\caption{The graph $Z_5$.}
\label{Z5Fig}
\end{figure}
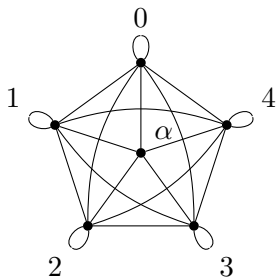

\begin{lem}
\label{colouredtoW}
\Hrec{(W_{2k+1},Z_{2k+1})} reduces to \Hrec{W_{2k+1}}. 
\end{lem}

\begin{lem}
\label{Ktocoloured}
\Hrec{K_{2k+1}} reduces to \Hrec{(W_{2k+1},Z_{2k+1})}.
\end{lem}

We prove Lemma~\ref{colouredtoW} in the next subsection and then we prove Lemma~\ref{Ktocoloured} in the subsection which follows it.

For readers familiar with the ``indicator construction'' in Hell and Ne\v{s}et\v{r}il's~\cite{Hcol} proof of the \Hcol{H} dichotomy, parts of our reductions will look familiar.  We mimic the classical reduction of \Hcol{K_{2k+1}} to \Hcol{C_{2k+1}}. Let $G$ be a graph and replace each edge $uv\in E(G)$ with a copy of $P_{2k}$, identifying the ends of the path with $u$ and $v$, to create a graph $G'$.  The key to the reduction is that for any distinct pair of vertices in $C_{2k+1}$ there is a homomorphism from $P_{2k}$ to $C_{2k+1}$ with the ends of the path mapping to the two prescribed vertices, but there is no homomorphism mapping the ends of the path to the same vertex. Thus, $G'\to C_{2k+1}$ if and only if $G\to K_{2k+1}$. 

However, this reduction does not work to reduce \Hcol{K_{2k+1}} to \Hcol{C_{2k+1}} as the homomorphisms from $P_{2k}$ to $C_{2k+1}$ do not admit the necessary reconfiguration sequences between them.\footnote{The fact that this reduction does not work should come as no surprise since \Hcol{C_{2k+1}} is solvable in polynomial time while \Hrec{K_{2k+1}} is PSPACE-complete.}  To add more flexibility, we may view the $C_{2k+1}$-colourings of $P_{2k+1}$ as $W_{2k+1}$-colourings which allows us to find the reconfiguration sequences that we require.  Unfortunately, this now becomes too flexible in the sense that we can reconfigure to $W_{2k+1}$-colourings which map the endpoints of $P_{2k}$ to the same vertex (which corresponds to allowing non-proper $K_{2k+1}$-colouring of $G$).  We eliminate unwanted homomorphisms from $P_{2k}$ to $W_{2k+1}$ from reconfiguration sequences using a second indicator $F_k(x,y)$. In what we believe is a novel modification of the indicator construction for reconfiguration, we exclude the unwanted homomorphisms by ensuring they are frozen (isolated vertices in $\bHom(F_k(x,y),W_{2k+1})$) and thus cannot appear in a reconfiguration sequence, or they are pendant vertices in $\bHom(G',W_{2k+1})$ and can be removed from any reconfiguration sequence.

\subsection{Freezing Gadgets and Proof of Lemma~\ref{colouredtoW}}

The following gadget will be used to keep control over the set of vertices which are mapped to $\alpha$. See Figure~\ref{freezingGadgetFig} for an illustration of the ``main parts'' of this gadget. 

\begin{defn}[Freezing Gadget]
Let $x$ and $y$ be vertices. Define $F_{k}(x,y)$ to be the graph with vertex set
\[\{x,y\}\cup\left\{z_0^x,\dots,z_{4k-3}^x\right\}\cup\left\{z_0^y,\dots,z_{4k-3}^y\right\}\cup\left\{b^x,b^y\right\}\cup \{w_0,\dots,w_{2k}\}\cup\{\alpha'\}\] 
where 
\begin{itemize}
\item $z_0^x\cdots z_{4k-3}^x$ and $z_0^y\cdots z_{4k-3}^y$ are cycles, 
\item $z_1^x$ and $z_{2k}^x$ are adjacent to $x$ and $z_1^y$ and $z_{2k}^y$ are adjacent to $y$, 
\item $z_0^x$ and $z_{2k-1}^x$ are adjacent to $b^x$ and $z_0^y$ and $z_{2k-1}^y$ are adjacent to $b^y$,
\item $y$ is adjacent to $b^x$ and $x$ is adjacent to $b^y$,
\item $w_0\cdots w_{2k}$ is a cycle, and
\item $\alpha'$ is joined to every vertex of $F_k(x,y)$, except for itself, $x,y$ and their neighbours.
\end{itemize}
\end{defn}

\begin{figure}[htbp]
\begin{tikzpicture}[scale=1.25]
\def\rightShift{2.75}
\def\upShift{0.0}
\node[smallblack,label={[label distance=-4pt]135:{$z^y_1$}}] (zy1) at (0,0){};
\node[smallblack,label=left:{$z^y_2$}] (zy2) at (0,-0.5){};
\node[label=left:{$\vdots$}](dots1) at (0,-1){};
\node[smallblack,label={[label distance=-10pt]225:{$z^y_{2k-1}$}}] (zy2k-1) at (0,-1.5){};
\node[smallblack,label={[label distance=-7pt]315:{$z^y_{2k}$}}] (zy2k) at (1.5,-1.5){};
\node[label=right:{$\vdots$}](dots2) at (1.5,-1){};
\node[smallblack,label=right:{$z^y_{4k-3}$}] (zy4k-3) at (1.5,-0.5){};
\node[smallblack,label={[label distance=-4pt]45:{$z^y_{0}$}}] (zy0) at (1.5,0){};
\node[smallblack,label=right:{$y$}] (y) at (0.75,-0.75){};
\draw[black] (zy1)--(zy2)--(zy2k-1)--(zy2k)--(zy4k-3)--(zy0)--(zy1)--(y)--(zy2k);

\node[smallblack,label={[label distance=-4pt]135:{$z^x_1$}}] (zx1) at (1.5+\rightShift,1.5+\upShift){};
\node[smallblack,label=left:{$z^x_2$}] (zx2) at (1.5+\rightShift,1+\upShift){};
\node[label=left:{$\vdots$}](dots3) at (1.5+\rightShift,0.5+\upShift){};
\node[smallblack,label={[label distance=-7pt]225:{$z^x_{2k-1}$}}] (zx2k-1) at (1.5+\rightShift,0+\upShift){};
\node[smallblack,label={[label distance=-6pt]315:{$z^x_{2k}$}}] (zx2k) at (3+\rightShift,0+\upShift){};
\node[label=right:{$\vdots$}](dots4) at (3+\rightShift,0.5+\upShift){};
\node[smallblack,label=right:{$z^x_{4k-3}$}] (zx4k-3) at (3+\rightShift,1+\upShift){};
\node[smallblack,label={[label distance=-4pt]45:{$z^x_{0}$}}] (zx0) at (3+\rightShift,1.5+\upShift){};
\node[smallblack,label=right:{$x$}] (x) at (2.25+\rightShift,0.75+\upShift){};
\draw[black] (zx1)--(zx2)--(zx2k-1)--(zx2k)--(zx4k-3)--(zx0)--(zx1)--(x)--(zx2k);

\node[smallblack,label={[label distance=-3pt]280:{$b^y$}}] (ay) at (2.75,-1.75){};
\node[smallblack,label={[label distance=-3pt]110:{$b^x$}}] (ax) at (0.25+\rightShift,1.75+\upShift){};

\draw [black] (ay) to[out=90,in=0] (zy0);
\draw [black] (ay) to[out=200,in=315]  (zy2k-1);
\draw [black] (ay) to[out=35,in=240]  (x);

\draw [black] (ax) to[out=270,in=180] (zx2k-1);
\draw [black] (ax) to[out=20,in=135]  (zx0);
\draw [black] (ax) to[out=215,in=60]  (y);

\end{tikzpicture}
\caption{The graph $F_k(x,y)\setminus \left\{w_0,\dots,w_{2k},\alpha'\right\}$.}
\label{freezingGadgetFig}
\end{figure}
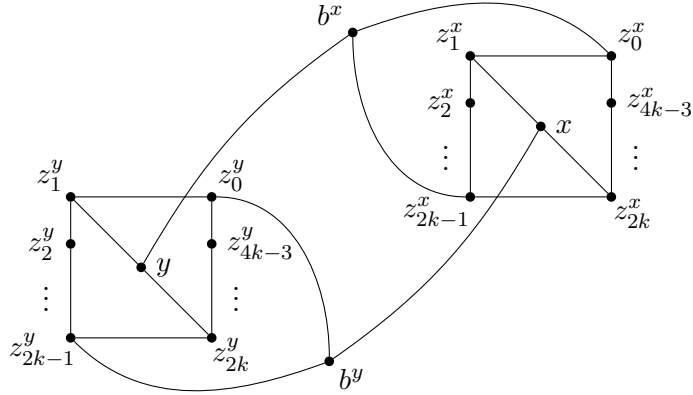

As $F_k(x,y)$ contains a copy of $W_{2k+1}$, the homomorphisms of $W_{2k+1} \to W_{2k+1}$ play a key role in our analysis.  Observe that under any homomorphism $\varphi: W_{2k+1} \to W_{2k+1}$, for $k \geq 2$, the outer cycle $W_{2k+1} \setminus \{ \alpha \}$ must map to an odd cycle every vertex of which is adjacent to $\varphi(\alpha)$.  If follows that $\varphi(\alpha) = \alpha$ and the outer cycle maps to the outer cycle, so $\varphi$ is frozen.  (For readers familiar with the concept of a \emph{core}, these observations are immediate since $W_{2k+1}$ is a core, any $\varphi: W_{2k+1} \to W_{2k+1}$ must be an automorphism.) Applying these arguments to the copy of $W_{2k+1}$ in $F_k(x,y)$ (not shown in the figure) we get the following.

\begin{obs}
\label{toalpha}
The only vertices that can map to $\alpha$ under a homomorphism $F_k(x,y) \to W_{2k+1}$ are $\alpha'$ (which must), $x$, $y$, and the neighbours of $x$ and $y$.  
\end{obs}

A second useful observation is the following.

\begin{obs}
\label{anyPair}
For every $c_1,c_2\in V(W_{2k+1})$ there exists a $W_{2k+1}$-colouring $f$ of $F_k(x,y)$ such that $f(x)=c_1$ and $f(y)=c_2$. 
\end{obs}

The next two lemmas illustrate the key properties of the freezing gadget which we will apply; first, a definition and a remark. 

\begin{defn}
Given an $H$-colouring $f$ of $G$, a vertex $v\in V(G)$ is said to be \emph{fixed} by $f$ if $f(v) = g(v)$ for every neighbour $g$ of $f$ in $\bHom(G,H)$.
\end{defn}

\begin{rem}
Fixed vertices and frozen vertices are similar, but not quite the same. A vertex $v \in V(G)$ is frozen by $f$ if $g(v) = f(v)$ for every $g$ in the same component of $\bHom(G,H)$ as $f$. Thus, if a vertex $v\in V(G)$ is frozen by $f$, then it is also fixed by $f$, but the converse does not hold. However, $f$ itself is frozen if and only if every vertex of $G$ is fixed by $f$.
\end{rem}

\begin{lem}
\label{alphaFrozen}
If $f$ is a $W_{2k+1}$-colouring of $F_k(x,y)$ with $f(x)=f(y)=\alpha$, then $f$ is frozen. 
\end{lem}

\begin{proof}
Let $f$ be such a homomorphism. Our goal is to show that every vertex of $F_k(x,y)$ is fixed by $f$. Since $\{w_0,\dots,w_{2k},\alpha'\}$ induces a copy of $W_{2k+1}$ in $F_k(x,y)$ and every homomorphism from $W_{2k+1}$ to itself is frozen, we have that every vertex of $\left\{w_0,\dots,w_{2k},\alpha'\right\}$ is frozen, and thus fixed. Moreover, every homomorphism from $W_{2k+1}$ to itself maps $\alpha$ to itself, and therefore $f(\alpha')=\alpha$. Thus, since we are assuming that $f(x)=f(y)=\alpha$, every vertex of $F_k(x,y)$ apart from $x,y$ and $\alpha'$ has a neighbour of colour $\alpha$, and it therefore mapped into $\{0,\dots,2k\}$. 

Consider the cycle $z_0^yz_1^y\cdots z_{2k-1}^y b^y$ of length $2k+1$. The cycle of length $2k+1$ has only one homomorphism to itself, up to composing with an automorphism. So, without loss of generality, we may assume that 
\[f\left(z_i^y\right)=i\quad \text{for } 0\leq i\leq 2k-1,\]
\[f\left(b^y\right)=2k.\]

Now, consider the cycle $z_{2k-1}^y z_{2k}^y\cdots z_{4k-3}^yz_0^y b^y$, which is also of length $2k+1$. We already know that $f(z_{2k-1}^y)=2k-1$, that $f(b^y)=2k$ and that $f(z_0^y)=0$. Applying uniqueness of the homomorphism from $C_{2k+1}$ to itself once again, we see that this implies 
\[f\left(z_{2k+i}^y\right)=2k-i-2\quad \text{for }0\leq i\leq 2k-3.\]
Thus, every vertex $u\in \left\{z_0^y,\dots,z_{4k-3}^y,b^y\right\}$ has neighbours of colours $\alpha$, $f(u)-1$ and $f(u)+1$. Therefore, every such vertex is fixed by $f$. Also, $y$ is joined to $z_1^y$ and $z_{2k}^y$ which are mapped by $f$ to colours $1$ and $2k-2$, respectively. The only vertex of $W_{2k+1}$ joined to these two colours is $\alpha$, and so $y$ is fixed by $f$. By symmetry, the vertices $z_0^x,\dots,z_{4k-3}^x,b^x$ and $x$ are also fixed and so $f$ is indeed frozen. 
\end{proof}

Lemma~\ref{notAlphaAlpha} below shows the $F_k(x,y)$ gadget admits reconfiguration sequences which are required in later proofs.  In order to prove the lemma, we use the concept of mixing.  A graph $G$ is \emph{$H$-mixing} if for any two $H$-colourings $f$ and $g$ of $G$, there is a reconfiguration sequence taking $f$ to $g$; that is, $\bHom(G,H)$ is connected.  The following two propositions are used in the proof; see~\cite[Propositions~3.12 and~7.12]{BrewsterNoel} for proofs. 

\begin{prop}
\label{treeMix}
Every tree $T$ is $C_{2k+1}$-mixing.
\end{prop}

\begin{prop}
\label{cycleMix}
For $r\geq3$, if $r\leq 2k$, then $C_{2r}$ is $C_{2k+1}$-mixing. 
\end{prop}

\begin{lem}
\label{notAlphaAlpha}
Let $c_1,c_1',c_2\in V(W_{2k+1})$ such that $(c_1,c_2)\neq (\alpha,\alpha)$ and $(c_1',c_2)\neq (\alpha,\alpha)$. If $f$ and $g$ are $W_{2k+1}$-colourings of $F_k(x,y)$ with $f(u)=g(u)$ for all $u\in \left\{w_0,\dots,w_{2k},\alpha'\right\}$ such that $(f(x),f(y))= (c_1,c_2)$ and $(g(x),g(y))= (c_1',c_2)$, then there exists a reconfiguration sequence $f_0f_1\cdots f_m$  and an index $0\leq j\leq m-1$ such that 
\begin{itemize}
\item $f_0=f$ and $f_m=g$,
\item $(f_i(x),f_i(y))=(c_1,c_2)$ for $0\leq i\leq j$, and
\item $(f_i(x),f_i(y))=(c_1',c_2)$ for $j+1\leq i\leq m$. 
\end{itemize}
\end{lem}

\begin{proof}
First, suppose that $c_1=c_1'$. In this case, we need to show that we can reconfigure $f$ to $g$ without changing the colour of $x$ or the colour of $y$. As a first step, if $c_1\neq \alpha$, then we change the colours of the neighbours of $x$ to $\alpha$ one by one (note that by Observation~\ref{toalpha}, if $c_1\neq \alpha$, then no such vertex has a neighbour of colour $\alpha$); otherwise, we leave the colours of the neighbours of $x$ unchanged.  Similarly, if $c_2\neq \alpha$, then we change the colours of the neighbours of $y$ to $\alpha$ one by one. Let $h$ be the resulting colouring.

Now, let $F'$ be the subgraph of $F_k(x,y)$ induced by the set of vertices $u\in V(F_k(x,y))\setminus\left\{x,y,w_0,\dots,w_{2k}\right\}$ such that $h(u)\neq\alpha$. By definition of $h$ and the fact $h$ and $g$ agree on $x$ and $y$, we see that each vertex of $F'$ is also mapped into $\{0,\dots,2k\}$ by $g$. By assumption, at least one of $c_1$ or $c_2$ must be different from $\alpha$. If exactly one of $c_1$ or $c_2$ is equal to $\alpha$, then $F'$ is the disjoint union of a cycle of length $4k-2$ and a path of length $4k-3$ and if, neither $c_1$ nor $c_2$ is equal to $\alpha$, then $F'$ is the disjoint union of four paths of length $2k-2$. In either case, by Propositions~\ref{treeMix} and~\ref{cycleMix}, every component of $F'$ is $C_{2k+1}$-mixing. Thus, there is a reconfiguration sequence taking $h$ to a colouring $h'$ such that the colour of each vertex outside of $V(F')$ remains constant throughout this sequence and $h'$ and $g$ agree on the vertices of $F'$. As a final step, we change the colour of each vertex which is a neighbour of $x$ or $y$, one by one, to match its colour under $g$. This is possible since $h'$ and $g$ agree on the neighbourhood of every such vertex. This completes the proof in the case that $c_1=c_1'$. 

Now, suppose that $c_1\neq c_1'$. In this case, all that we need to do is to show that $f$ can be reconfigured to a $W_{2k+1}$-colouring $f'$ such that $f'(x)=c_1'$ and $f'(y)=c_2$ without changing the colour of $y$ during the sequence. If this is true, then we can simply apply the result of the previous case to reconfigure $f'$ to $g$ and we will be done. 

Consider first the case that $c_1,c_1'\neq \alpha$. In this case, no neighbour of $x$ is adjacent to a vertex which is mapped to $\alpha$ by $f$. Thus, we can change the colours of the neighbours of $x$ to $\alpha$, one by one, and then change the colour of $x$ from $c_1$ to $c_1'$. Letting $f'$ be the resulting colouring, we see that we are done in this case. 

Next, suppose that $c_1=\alpha$ and $c_1'\neq \alpha$. By hypothesis, this implies that $c_2\neq\alpha$. We begin by changing the colours of the neighbours of $y$ to $\alpha$ one by one and letting $h$ be the resulting $W_{2k+1}$-colouring. Let $F'$ be the subgraph of $F_{k}(x,y)$ induced by the set of vertices $u\in V(F_k(x,y))\setminus\left\{y,w_0,\dots,w_{2k}\right\}$ such that $h(u)\neq\alpha$. Then the components of $F'$ are the cycle $z_0^x\cdots z_{4k-3}^x$ and the path $z^y_{2k+1}\cdots z_{4k-3}^yz_0^y b^y z_{2k-1}^y\cdots z_2^y$. By Proposition~\ref{cycleMix}, by only changing colours of vertices on $F'$, we can reconfigure $h$ to a $W_{2k+1}$-colouring $h'$ such that 
\[h'(z_0^x) = c_1'+2,\]
\[h'\left(z_i^x\right) = c_1'+i\quad \text{for }1\leq i\leq 2k,\]
\[h'\left(z_{2k+i}^x\right) = c_1'+2k-i\quad \text{for }1\leq i\leq 2k-3,\]
\[h'(b^y)=c_1'+1.\]
Now, we are simply done by changing the colour of $x$ from $c_1=\alpha$ to $c_1'$. 

Finally, suppose that $c_1\neq \alpha$ and $c_1'=\alpha$. By reversing the roles of $f$ and $g$, and applying the previous case,
we are done.  (Equivalently, $\bHom(F_k(x,y),W_{2k+1})$ is an undirected graph and we can follow the path from the previous case in the opposite direction.)
\end{proof}

We are now in position to prove Lemma~\ref{colouredtoW}.

\begin{proof}[Proof of Lemma~\ref{colouredtoW}]
Let $G=(G_1,G_2)$ be an edge-coloured graph and let $f$ and $g$ be two $(W_{2k+1},Z_{2k+1})$-colourings of $G$. Let $G'$ be the graph obtained from $G_1$ by adding the gadget $F_k(x,y)$ for every $xy\in E(G_2)$ disjointly from $G_1$ except at $x$ and $y$. Using Observation~\ref{anyPair}, we can let $f'$ and $g'$ be any $W_{2k+1}$-colourings of $G'$ obtained by extending $f$ and $g$ to the vertices of $V(G')\setminus V(G)$, respectively. Clearly, this construction can be completed in polynomial time and $V(G')=O\left(|V(G)|+|E(G_2)|\right)$. We claim that $f$ reconfigures to $g$ as $(W_{2k+1},Z_{2k+1})$-colourings of $G$ if and only if $f'$ reconfigures to $g'$ as $W_{2k+1}$-colourings of $G'$. 

First, suppose that $f'$ reconfigures to $g'$ and let $f_0',\dots,f_m'$ be any reconfiguration sequence taking $f'$ to $g'$. For $0\leq i\leq m$, let $f_i$ be the restriction of $f_i'$ to $V(G)$. Clearly $f_0=f$ and $f_m=g$ and so we are done if $f_i$ is a $(W_{2k+1},Z_{2k+1})$-colouring of $G$ for $1\leq i\leq m$. Since $G_1$ is a subgraph of $G'$, we have that $f_i$ is a $W_{2k+1}$-colouring of $G_1$ and so the only way in which it could fail to be a $(W_{2k+1},Z_{2k+1})$-colouring is if there exists $xy\in E(G_2)$ such that $f_i(x)=f_i(y)=\alpha$. For any such pair, the gadget $F_k(x,y)$ is present in $G'$. Thus, by Lemma~\ref{alphaFrozen}, the vertices $x$ and $y$ are frozen by $f_i'$. In particular, this implies that $f(x)=f_0'(x)=\alpha$ and $f(y)=f_0'(y)=\alpha$, which contradicts the assumption that $f$ is a $(W_{2k+1},Z_{2k+1})$-colouring of $G$. This completes the proof of this direction.

Now, we suppose that $f$ reconfigures to $g$ and show that $f'$ reconfigures to $g'$. We may assume that $f$ and $g$ differ on only one vertex, say $x\in V(G)$, from which the general case will follow by induction on the length of the reconfiguration sequence taking $f$ to $g$. 

As a first step, let us show that $f'$ reconfigures to a $W_{2k+1}$-colouring $h'$ of $G'$ such that $h'(u)=g(u)$ for every $u\in V(G)$ (including $x$). Given a vertex $y$ such that $xy\in E(G_2)$, let $c_1:=f(x)$, $c_1':=g(x)$ and $c_2:=f(y)$ (note that $g(y)=c_2$ since we are assuming that $f$ and $g$ differ only on $x$). Since $f$ and $g$ are $(W_{2k+1},Z_{2k+1})$-colourings of $G$, we have that $(c_1,c_2), (c_1',c_2)\neq (\alpha,\alpha)$. Thus, by Lemma~\ref{notAlphaAlpha}, there is a reconfiguration sequence of $(W_{2k+1},Z_{2k+1})$-colourings of $F_k(x,y)$ taking the restriction of $f'$ to $F_k(x,y)$ to a colouring which maps $x$ to $c_1'$ and $y$ to $c_2$ such that the colour of $y$ does not change during this sequence and the colour of $x$ changes exactly once. For every $y$ such that $xy\in E(G_2)$, one at a time, we perform the ``first half'' of this reconfiguration sequence, stopping before the colour of $x$ changes. After this, we can safely change the colour of $x$ from $f(x)$ to $g(x)$, thereby obtaining the desired $W_{2k+1}$-colouring $h'$ of $G'$. 

Now, we have that $h'$ and $g'$ agree on all vertices of $G$, but that they may differ on some of the the freezing gadgets. However, since $g'$ is a $(W_{2k+1},Z_{2k+1})$-colouring (i.e. it does not map both sides of any edge of $G_2$ to $\alpha$), we can apply Lemma~\ref{notAlphaAlpha} (in the case $c_1=c_1'$) to each freezing gadget one by to make the colourings of these gadgets match their colourings under $g'$ without changing the colour of any vertex of $V(G)$. Thus, $f'$ reconfigures to $h'$ which reconfigures to $g'$ and the proof is complete. 
\end{proof}

\subsection{Proof of Lemma~\ref{Ktocoloured}}

Finally, we prove Lemma~\ref{Ktocoloured}, which completes the proof of Theorem~\ref{wheelThm}.

\begin{proof}[Proof of Lemma~\ref{Ktocoloured}]
Let $G$ be a graph and let $\varphi$ and $\psi$ be homomorphisms from $G$ to $K_{2k+1}$. We construct an instance of \Hrec{(W_{2k+1},Z_{2k+1})} in five steps, as follows:

\begin{step}
Let $G^*$ be the graph obtained from $G$ by subdividing each edge of $G$ exactly $2k-2$ times. Vertices of $V(G)$ are called \emph{original} vertices and vertices of $V(G^*)\setminus V(G)$ are called \emph{subdivision} vertices.  (That is, replace each edge with a path of length $2k-1$.)
\end{step}

\begin{step}
\label{lock}
For each original vertex $v$, add $2k$ new vertices, say $\ell_1^v,\dots,\ell_{2k}^v$ such that $v\ell_1^v\cdots \ell_{2k}^v$ forms a cycle. These new vertices are called \emph{locking} vertices. Let $\tilde{G}$ be the resulting graph. 
\end{step}

\begin{step}
\label{Wstep}
Add a set $W$ of $2k+2$ vertices disjoint from $V(G^*)$ such that $W$ induces a copy of $W_{2k+1}$ and every vertex which is not adjacent to an original vertex in $\tilde{G}$ is joined to the vertex of $W$ corresponding to vertex $\alpha$. Let $G_1'$ be the resulting graph.
\end{step}

\begin{step}
Extend $\varphi$ and $\psi$, respectively, to the subdivision vertices, the locking vertices, and $W$ to obtain $\varphi'$ and $\psi'$ mapping $G'_1$ to $W_{2k+1}$.  We may assume $\varphi'$ and $\psi'$ agree on $W$ and do not map any vertex of $V(G'_1)\setminus W$ to $\alpha$. The existence of these extensions follows from the fact that, for any two distinct vertices of $C_{2k+1}$, there is a homomorphism from $P_{2k}$ to $C_{2k+1}$ mapping the endpoints to these vertices.
\end{step}

\begin{step}
\label{freezingStep}
Let $G_2'$ be the graph with vertex set $V(G_1')$ where $xy$ is an edge of $G_2'$ if there exists distinct original vertices $u,v$ such that $xu\in E(G^*)$ and $y\in\left\{\ell_1^v,\ell_{2k}^v\right\}$.
\end{step}

It is clear that these steps can be performed in polynomial time and that the resulting edge-coloured graph $G'=(G_1',G_2')$ satisfies $|V(G')|=O\left(|V(G)|+|E(G)|\right)$. Also, since $\varphi'$ and $\psi'$ do not use colour $\alpha$ on $V(G')\setminus W$, it is clear that they are $(W_{2k+1},Z_{2k+1})$-colourings of $G'$. Our goal is to show that $\varphi$ reconfigures to $\psi$ (as $K_{2k+1}$-colourings) if and only if $\varphi'$ reconfigures to $\psi'$ (as $(W_{2k+1},Z_{2k+1})$-colourings). 

First suppose that $\varphi'$ reconfigures to $\psi'$ and let $f_0',\dots,f_m'$ be any reconfiguration sequence taking $\varphi'$ to $\psi'$. For $0\leq i\leq m$, let $f_i$ be the restriction of $f_i'$ to $V(G)$. Clearly, we have $f_0=\varphi$ and $f_m=\psi$ and that $f_i$ differs from $f_{i+1}$ on at most one vertex for $0\leq i\leq m-1$. Since every original vertex is joined to a vertex of $W$ of colour $\alpha$, we have  that  $f_i(v)\in \{0,\dots,2k\}$ for all $v\in V(G)$ and $0\leq i\leq m$. Of course, if each of $f_0,\dots,f_m$ is a $K_{2k+1}$-colouring, then we have found a reconfiguration sequence taking $\varphi$ to $\psi$ and we are done. The following claim allows us to handle the possibility that some of the $f_i$ map adjacent vertices of $G$ to the same colour. 

\begin{claim}
\label{restrictProper}
For $1\leq i\leq m-1$, if $f_i$ is not a $K_{2k+1}$-colouring, then there exists a vertex $v\in V(G)$ such that $f_{i-1}$ and $f_{i+1}$ agree on $V(G)\setminus\{v\}$. 
\end{claim}

\begin{proof}
Suppose that $f_i$ is not a $K_{2k+1}$-colouring and let $uv\in E(G)$ such that $f_i(u)=f_i(v)$. (By Lemma~\ref{alphaFrozen} $f_i(v) \neq \alpha$.) Observe that there is no homomorphism from $P_{2k}$ to $C_{2k+1}$ mapping the endpoints to the same vertex. Therefore, at least one vertex, say $x$, on the copy of $P_{2k}$ in $G'_1$ obtained from subdividing the edge $uv$ must be mapped to colour $\alpha$ by $f_i$. By Step~\ref{Wstep}, we have that $x$ must be either a neighbour of $u$ or a neighbour of $v$. Without loss of generality, $x$ is a neighbour of $v$. 

Now recall that, for every original vertex $w\neq v$ and $y\in \left\{\ell_1^w,\ell_{2k}^w\right\}$, we have that $xy\in E(G_2')$. Since $x$ maps to $\alpha$ and $f_i'$ is a $(W_{2k+1},Z_{2k+1})$-colouring of $G'$, we must have that $f_i'(y)\in \{0,\dots,2k\}$. Since all of the vertices $\ell_2^w,\dots,\ell_{2k-1}^w$ are adjacent to a vertex of colour $\alpha$, this implies that every vertex of the cycle $w\ell_1^w\cdots \ell_{2k}^w$ is mapped into $\{0,\dots,2k\}$. Since there is a unique homomorphism from $C_{2k+1}$ to itself, up to automorphism, this implies that every $w\in V(G)\setminus\{v\}$ is adjacent to a vertex of colour $\alpha$, a vertex of colour $f_i(w)+1$ and a vertex of colour $f_i(w)-1$. Thus, every such vertex is fixed by $f_i'$, which proves the claim. 
\end{proof}

Thus by the claim, $f_0, \dots, f_{i-1}, f_{i+1}, \dots, f_m$ is a reconfiguration sequence without the improper colouring $f_i$. Let  $g_0,\dots,g_t$ be the subsequence of $f_0,\dots,f_m$ consisting only of the mappings $f_j$ which are $K_{2k+1}$-colourings. By Claim~\ref{restrictProper}, we have that $g_i$ and $g_{i+1}$ differ on at most one vertex for $0\leq i\leq t-1$, and so this sequence certifies that $\varphi$ reconfigures to $\psi$. This completes one direction of the proof. 

Suppose now that there is a reconfiguration sequence $f_0,\dots, f_m$ taking $\varphi$ to $\psi$. Our goal is to show that $\varphi'$ reconfigures to $\psi'$.  We may assume that $m=1$ since the general case will follow easily by induction on $m$. So, we assume that there exists a unique vertex $v\in V(G)$ such that $\varphi(v)\neq \psi(v)$. 

By construction, no vertex of $G' \backslash W$ is mapped to $\alpha$ by $\varphi'$. So, as a first step, we may change the colour of each neighbour of $v$ to $\alpha$, one at a time. Note that this preserves the property of being a $(W_{2k+1},Z_{2k+1})$-colouring. At this point, we can safely change the colour of $v$ from $\varphi(v)$ to $\psi(v)$. Thus, we have reached a $(W_{2k+1},Z_{2k+1})$-colouring, say $\gamma'$, of $G'$ which maps every vertex of $V(G)$ to its colour under $\psi$. 

The final step is to show that $\gamma'$ reconfigures to $\psi'$. By construction, $\gamma'$ maps every neighbour of $v$, including $\ell_1^v$ and $\ell_{2k}^v$, to $\alpha$. So, we can apply Proposition~\ref{treeMix} to reconfigure the colouring of the path $\ell_{2}^v\cdots\ell_{2k-1}^v$ to match its colouring under $\psi'$. After this, we can change the colours of $\ell_1^v$ and $\ell_{2k}^v$ to match their colours under $\psi'$. At this point, none of the vertices of the form $\ell_1^u$ or $\ell_{2k}^u$ for $u\in V(G)$ are mapped to $\alpha$. Thus, by definition of $G_2'$, we can now change the colour of every subdivision vertex $x$ which is in the neighbourhood of some original vertex to $\alpha$, one by one. The set of subdivision vertices which are not currently mapped to $\alpha$ induces a disjoint union of paths and so we can apply Proposition~\ref{treeMix} to change the colouring of these vertices to match its colouring under $\varphi'$. We can now change the colour of every subdivision vertex $x$ which is the neighbour of an original vertex from $\alpha$ to $\psi'(x)$. Now, we have reached a $(W_{2k+1},Z_{2k+1})$-colouring which only possibly differs from $\psi'$ on vertices of the form $\ell_i^u$ for $u\in V(G)$. We complete the reconfiguration sequence by going through the vertices $u\in V(G)$, one by one, changing the colours of $\ell_1^u$ and $\ell_{2k}^u$ to $\alpha$, applying Proposition~\ref{treeMix} to make the colouring of the path $\ell_2^u\cdots \ell_{2k-1}^u$ match its colouring under $\psi'$, and then changing the colours of $\ell_1^u$ and $\ell_{2k}^u$  to match $\psi'$ as well. Thus, $\varphi'$ reconfigures to $\psi'$ and we are done. 
\end{proof}

\section{Frozen Homomorphisms and Constraint Satisfaction}
\label{frozenSection}

\subsection{The Structure of Thermal Graphs and Frozen Homomorphisms}

Our goal in this section is to prove a dichotomy theorem for \Hfrz{H}. We begin by building up a sequence of basic observations regarding the structure of thermal graphs and frozen homomorphisms. We will mostly focus on frozen homomorphisms to a connected graph $F$ (which we think of as being a component of $H$). Given a graph $F$ and $v\in V(F)$, let $N_F(v)$ be the set of neighbours of $v$ (in particular, if $v$ has a loop, then $v\in N_F(v)$). 

\begin{defn}
Given a connected graph $F$ and a set $S\subseteq V(F)$, we say that a vertex $\alpha\in S$ is \emph{redundant} for $S$ if there exists a vertex $\beta\in V(F)\setminus\{\alpha\}$ such that
\[N_F(\alpha)\cap S\subseteq N_F(\beta)\cap S.\]
\end{defn}

\begin{defn}
Given a connected graph $F$, let $S_F$ be a (possibly empty) subset of $V(F)$ such that no vertex in $S_F$ is redundant for $S_F$ and, subject to this, $S_F$ is maximal under subset inclusion. 
\end{defn}

The following lemma implies that $S_F$ is, in fact, unique.

\begin{lem}
\label{SFunique}
If $F$ is a connected graph and $T\subseteq V(F)$ is a set such that no element of $T$ is redundant for $T$, then $T\subseteq S_F$. 
\end{lem}

\begin{proof}
Suppose not and let $S:=S_F\cup T$. By maximality of $S_F$, there must exist $\alpha\in S$ and $\beta\in V(F)\setminus\{\alpha\}$ such that
\[N_F(\alpha)\cap S \subseteq N_F(\beta)\cap S.\]
In particular, this implies that both 
\[N_F(\alpha)\cap S_F \subseteq N_F(\beta)\cap S_F\]
and 
\[N_F(\alpha)\cap T \subseteq N_F(\beta)\cap T.\]
This either contradicts the choice of $S_F$ or the choice of $T$, depending on whether $\alpha\in S_F$ or $\alpha\in T$. Thus, the lemma is proved.
\end{proof}

As a corollary of Lemma~\ref{SFunique}, we get that any frozen homomorphism from a graph $G$ to a connected graph $F$ must map into the set $S_F$. 

\begin{cor}
\label{frozenGoesToS}
Let $F$ be a connected graph. Then for every frozen homomorphism $f$ from a graph $G$ to $F$, we have $f(V(G))\subseteq S_F$. 
\end{cor}

\begin{proof}
By Lemma~\ref{SFunique}, if $f(V(G))\nsubseteq S_F$, then there exists a vertex $\alpha\in f(V(G))$ which is redundant for $f(V(G))$. Let $u$ be a vertex with $f(u)=\alpha$ and let $\beta\in V(F)\setminus\{\alpha\}$ such that $N_F(\alpha)\cap f(V(G))\subseteq N_F(\beta)\cap f(V(G))$. Then the function $g:V(G)\to V(F)$ defined by 
\[g(v):=\left\{\begin{array}{ll}\beta & \text{if }v=u,\\
f(v) &\text{otherwise}.\end{array}\right.\]
is a homomorphism from $G$ to $F$ which differs from $f$ only on $u$, contradicting the assumption that $f$ is frozen.  
\end{proof}

Next, we obtain a characterisation of thermal graphs.

\begin{cor}
\label{thermalEmpty}
A connected graph $F$ is thermal if and only if $S_F=\emptyset$. 
\end{cor}

\begin{proof}
If $F$ is not thermal, then there exists a graph $G$ with $V(G)\neq\emptyset$ such that there is a frozen homomorphism $f$ from $G$ to $F$. By Corollary~\ref{frozenGoesToS}, we have $f(V(G))\subseteq S_F$, which implies that $S_F\neq\emptyset$. 

Now, suppose that $S_F\neq\emptyset$ and let $F'$ be the subgraph of $F$ induced by $S_F$. Let $f:V(F')\to V(F)$ be defined by $f(\alpha)=\alpha$ for all $\alpha\in V(F')$. If $f$ is not frozen, then there exists a vertex $\alpha\in S_F$ which can change its colour from $\alpha$ to some $\beta\neq\alpha$. However, in order for this to be possible, we need $N_F(\alpha)\cap S_F\subseteq N_F(\beta)\cap S_F$, which implies that $\alpha$ is redundant for $S_F$. This contradiction completes the proof. 
\end{proof}

In other words, Corollary~\ref{thermalEmpty} says that a connected graph $F$ is thermal if and only if, for every subset $S$ of  $V(F)$, there exists a vertex $\alpha\in S$ and a vertex $\beta \in V(F)\setminus\{\alpha\}$ such that $N_F(\alpha)\cap S\subseteq N_F(\beta)\cap S$. This is reminiscent of (but different from) the notion of  ``dismantlability'' which appears in the study of pursuit games on graphs (see, e.g., Nowakowski and Winkler~\cite{Winkler}). A connected graph $F$ is said to be \emph{dismantlable} if, for every subset $S$ of $V(F)$, there exists distinct $\alpha,\beta\in S$ such that $\alpha\beta\in E(F)$ and $N_F(\alpha)\cap S\subseteq N_F(\beta)\cap S$.  

We remark that, given Lemma~\ref{SFunique} and Corollary~\ref{thermalEmpty}, it is easy to decide whether a connected graph $F$ is thermal in polynomial time. Simply begin by initialising $S\leftarrow V(F)$ and, while there exists a vertex $\alpha\in S$ which is redundant for $S$, set $S\leftarrow S\setminus\{\alpha\}$. The graph $F$ is thermal if and only if the set $S$ eventually becomes empty. The running time of this algorithm is clearly  polynomial in $|V(F)|$.

\begin{defn}
Given a connected graph $F$ and a vertex $\alpha\in V(F)$, we say that a set $D\subseteq S_F$ is \emph{distinguishing} for $\alpha$ if $\alpha$ is the unique vertex of $V(F)$ such that $D\subseteq N_F(\alpha)$.
\end{defn}

Building on the Corollary~\ref{frozenGoesToS}, we prove the following. 

\begin{lem}
\label{distinguishingLem}
Let $F$ be a connected graph. Then a homomorphism from a graph $G$ to $F$ is frozen if and only if  for all $v\in V(G)$ the set $f(N_G(v))$ is distinguishing for $f(v)$.
\end{lem}

\begin{proof}
Suppose first that $f$ is frozen. By Corollary~\ref{frozenGoesToS}, we know that $f(N_G(v))\subseteq S_F$ for all $v\in V(G)$. Clearly, $f(N_G(v))\subseteq N_F(f(v))$ since $f$ is a homomorphism. If there exists some $\beta\neq f(v)$ such that $f(N_G(v))\subseteq N_F(\beta)$, then we can change the colour of $v$ from $f(v)$ to $\beta$, contradicting the fact that $f$ is frozen. So, $f(N_G(v))$ is distinguishing for $f(v)$. 

Now, suppose that $f$ is a homomorphism such that $f(N_G(v))$ is distinguishing for $f(v)$ for every $v\in V(G)$. If $f$ were not frozen, then we could change the colour of some $v\in V(G)$ from $f(v)$ to some $\beta\neq f(v)$. However, both $f(v)$ and $\beta$ would have to be adjacent to every colour in $f(N_G(v))$, contradicting the fact that this set is distinguishing for $f(v)$. The result follows. 
\end{proof}

As the next lemma demonstrates, one trivial example of a distinguishing set is a set of the form $N_F(\alpha)\cap S_F$ where $\alpha\in S_F$. 

\begin{lem}
\label{nbhdDistinguishes}
Let $F$ be a connected graph and let $\alpha\in S_F$. Then the set $N_F(\alpha)\cap S_F$ is distinguishing for $\alpha$.
\end{lem}

\begin{proof}
If not, then there must be a vertex $\beta\in V(F)$ such that $N_F(\alpha)\cap S_F\subseteq N_F(\beta)$ which implies that $N_F(\alpha)\cap S_F\subseteq N_F(\beta)\cap S_F$ and so $\alpha$ is redundant for $S_F$. This is a contradiction.
\end{proof}

The following lemma provides an important distinction between $K_1$ and $K_2$ and other non-thermal graphs. 

\begin{lem}
\label{notK1K2}
If $F$ is a connected non-thermal graph with at least three vertices and $D\subseteq S_F$ is distinguishing for some $\alpha\in V(F)$, then $|D|\geq2$. 
\end{lem}

\begin{proof}
Clearly, since $F$ has more than one vertex, the empty set is not distinguishing for any $\alpha\in V(F)$. So, suppose that there is a set of cardinality one, say $\{\delta\}$ where $\delta \in S_F$, which is distinguishing for some $\alpha\in V(F)$. This means that $\alpha$ is the unique neighbour of $\delta$. 

Since $F$ is connected and has at least three vertices, we see that $\alpha$ must have a neighbour $\beta\notin\{\delta,\alpha\}$. We have $N_F(\delta)\cap S_F$ is either equal to $\emptyset$ or $\{\alpha\}$ depending on whether or not $\alpha\in S_F$. In either case, $N_F(\delta)\cap S_F$ is contained in $N_F(\beta)\cap S_F$. Thus, $\delta$ is redundant for $S_F$ which is a contradiction. 
\end{proof}

\subsection{General Dichotomy Theorem for \texorpdfstring{\Hfrz{\boldsymbol{H}}}{Frozen H-Colouring}}

Using the terminology of the previous subsection, we can now state the general dichotomy theorem for \Hfrz{H}. First, a definition.

\begin{defn}
Given a graph $H$, the \emph{freezer} of $H$ is the subgraph of $H$ induced by the union of the sets $S_F$ over all components $F$ of $H$. 
\end{defn}

To be clear, when we say that a graph is bipartite, we mean that it is loop-free and contains no odd cycle. 

\begin{thm}
\label{frozenThmGeneral}
\Hfrz{H} is solvable in polynomial time if either
\begin{itemize}
\item every component of $H$ is either thermal or contains at most two vertices, 
\item $H$ contains a component isomorphic to $K_2$ and the freezer of $H$ is bipartite, or
\item $H$ contains a component consisting of a single reflexive vertex.
\end{itemize}
Otherwise, \Hfrz{H} is NP-complete. 
\end{thm}

Clearly, Theorem~\ref{frozenThmConn} is equivalent to  Theorem~\ref{frozenThmGeneral} in the case that $H$ is connected. The polynomial side of Theorem~\ref{frozenThmGeneral} is easy, as we show now.

\begin{proof}[Proof of the polynomial side of Theorem~\ref{frozenThmGeneral}]
Let $G$ be any graph such that $V(G)\neq \emptyset$. Clearly, no frozen $H$-colouring of $G$ can map any vertex of $G$ into a thermal component of $H$. Therefore, $G$ admits a frozen $H$-colouring if and only if $G$ admits a frozen $H'$-colouring where $H'$ is obtained from $H$ by deleting all thermal components. So, from here forward, we can assume that $H$ has no thermal components. 

Suppose that $H$ contains an isolated vertex. If this isolated vertex is the only vertex of $H$, then $G$ admits a frozen $H$-colouring if and only if $G$ has no edges, and so the \Hfrz{H} problem is clearly solvable in polynomial time. If $H$ contains an isolated vertex and has more than one vertex, then $G$ admits a frozen $H$-colouring if and only if $G$ admits a frozen $H'$-colouring where $H'$ is obtained from $H$ by deleting all isolated vertices. Also, if $H$ has a component which is just a single reflexive vertex $x$, then $G$ admits a frozen $H$-colouring if and only if either $x$ is the only vertex of $H$ or $G$ has no isolated vertices. So, from here forward, we can assume that $H$ has no components $F$ with $|V(F)|=1$.

Thus, we can assume that every component of $H$ has at least two vertices and that $H$ has no thermal components. By hypothesis, this implies that $H$ contains a component isomorphic to $K_2$ and that the freezer of $H$ is bipartite. We claim that $G$ admits a frozen $H$-colouring if and only if $G$ is bipartite and contains no isolated vertices (which are all conditions which can clearly be checked in polynomial time). First, if $G$ is bipartite and contains no isolated vertices, then any homomorphism from $G$ to $K_2$ is frozen. On the other hand, if $G$ admits a frozen $H$-colouring, then $G$ cannot have an isolated vertex (as $H$ has at least two vertices). Also, by applying Corollary~\ref{frozenGoesToS} to every component of $H$, we see that the image of any frozen homomorphism from $G$ to $H$ must be contained in the vertex set of the freezer of $H$. Therefore, since the freezer of $H$ is  bipartite, $G$ must be bipartite too. This completes the proof. 
\end{proof}

\subsection{Relational Structures and Polymorphisms}

Before moving on to the proof of Theorem~\ref{frozenThmGeneral}, we 
pause to introduce a tool from the theory of constraint satisfaction
problems which is useful in showing that problems are NP-complete.

\begin{defn}
For an integer $k \geq 1$, a \emph{\kgraph~}$\mathcal{H}$
consists of a finite set $V = V(\mathcal{H})$ and a $k$-ary 
relation $R(\mathcal{H}) \subseteq V^k$. 
\end{defn}

A \kgraph~ is an example of a relational structure.  Most of
the following defintions are simplified versions of more general
definitions that are well known for relational structures.  

\begin{defn}
Given two \kgraphs~ $\mathcal{G}$ and $\mathcal{H}$, 
a \emph{homomorphism} from $\mathcal{G}$ to $\mathcal{H}$ is a function $f:V(\mathcal{G})\to V(\mathcal{H})$ such that 
 \[(r_1,\dots,r_k)\in R(\mathcal{G})\Rightarrow (f(r_1),\dots,f(r_k))\in R({\mathcal{H}}).\] 
\end{defn}

Given a \kgraph~ $\mathcal{H}$, the following decision problem
is known as the \emph{constraint satisfaction problem} for $\mathcal{H}$ and is denoted \CSP{\mathcal{H}}: 
\begin{itemize}
\item[] \textbf{Instance:} A \kgraph~ $\mathcal{G}$. 
\item[] \textbf{Question:} Does there exist a homomorphism from
  $\mathcal{G}$ to $\mathcal{H}$?
\end{itemize}
The famous dichotomy conjecture for constraint satisfaction problems of Feder and Vardi~\cite{CSP} says that, for every relational structure $\mathcal{H}$ (a well known
generalisation of a \kgraph), the problem \CSP{\mathcal{H}} is either solvable in polynomial time or is NP-complete. This problem has been open for nearly 20 years and motivated the discovery of numerous important connections between  theoretical computer science,  combinatorics and universal algebra. Very recently, several groups have announced proofs of the dichotomy conjecture for CSPs~\cite{CSPproof,Bulatov2017,Zhuk2017,Delic2017} (one of which has now been withdrawn; see~\cite{Willard2017}). One of the key discoveries in the study of the complexity of constraint satisfaction problems is that the complexity of \CSP{\mathcal{H}} is intimately tied to the types of polymorphisms (which we define next) that $\mathcal{H}$ possesses. 

\begin{defn}
Let $\mathcal{H}$ be a \kgraph. 
An \emph{$m$-ary polymorphism} of $\mathcal{H}$ is a function $\varphi:V(\mathcal{H})^m\to V(\mathcal{H})$ which preserves $R(\mathcal{H})$; that is, such that  for any set of $m$ $k$-tuples 
\[(r_{1,1},\dots,r_{1,k}),(r_{2,1},\dots,r_{2,k})\dots,(r_{m,1},\dots,r_{m,k})\in R\]
we have 
\[\left(\varphi(r_{1,1},\dots,r_{m,1}),\varphi(r_{1,2},\dots,r_{m,2}),\dots,\varphi(r_{1,k},\dots,r_{m,k})\right)\in R\] 
\end{defn}

In the proof of Theorem~\ref{frozenThmGeneral} in the next subsection, we will apply a consequence of the following theorem. It has been observed be several people, (see Barto and Stanovsk\'y~\cite{Barto}) that this theorem, or actually a more general version stated for relational structures, follows from Barto, Kosik, and Niven~\cite{BKN} by the proof of a similar statement by Siggers in~\cite{Siggers10}. 

\begin{thm}[\cite{Barto,Siggers10}]
\label{markThm}
If $\mathcal{H}$ is a \kgraph~ which does not admit a $4$-ary polymorphism $\varphi$ such that
\[\varphi(a,r,e,a)=\varphi(r,a,r,e)\]
for all $a,r,e\in V(\mathcal{H})$, then $CSP(\mathcal{H})$ is NP-complete.

\end{thm}

Say that a \kgraph~ $\mathcal{H}$ is  \emph{totally symmetric} if whenever $(r_1,\dots,r_k)\in R(\mathcal{H})$ and $(r_1',\dots,r_k')$ is a $k$-tuple such that $\{r_1',\dots,r_k'\}=\{r_1,\dots,r_k\}$ we have $(r_1',\dots,r_k')\in R(\mathcal{H})$. A \emph{constant tuple} is an element of $V^k$ for some $k$ of the form $(v,\dots,v)$.  We use Theorem~\ref{markThm} to prove the following lemma which will be applied in the proof of Theorem~\ref{frozenThmGeneral}. This lemma is of independent interest; indeed, we recently found out that it has been proved independently, with quite a different proof, by Ham and Jackson \cite{HJ}. 

\begin{lem}
\label{symmetricLem}
If $\mathcal{H}$ is a non-empty totally symmetric \kgraph, for $k \geq 3$,  containing no  constant tuple, then \CSP{\mathcal{H}} is NP-complete.
\end{lem}

\begin{proof}
Let $m$ be the minimum of $|\{r_1,\dots, r_k\}|$ over all $(r_1,\dots,r_k)\in R(\mathcal{H})$; i.e. $m$ is the smallest number of distinct symbols which appear in an element of the relation $R(\mathcal{H})$. Up to relabelling the elements of $V(\mathcal{H})$, we can assume that there is $(r_1,\dots,r_k)\in R(\mathcal{H})$ such that $\{r_1,\dots,r_k\}=\{1,\dots,m\}$.  Since $R(\mathcal{H})$ contains no constant tuple, we know that $m\geq2$. 

Now, suppose to the contrary that $\mathcal{H}$ is not NP-complete. Then, by Theorem~\ref{markThm}, there exists a $4$-ary polymorphism on $\mathcal{H}$ such that
\begin{equation}\label{area}\varphi(a,r,e,a)=\varphi(r,a,r,e)\end{equation}
for all $a,r,e\in V(\mathcal{H})$. If $k>m$, then, since $R(\mathcal{H})$ is totally symmetric, we have that $R(\mathcal{H})$ contains the following four $k$-tuples (written as column vectors for convenience when applying $\varphi$):
\[\begin{array}{cccc}\begin{bmatrix}
           2 \\
           1 \\
           2 \\
           3 \\
           4\\
           \vdots \\
           m\\
           \vdots \\
           m
         \end{bmatrix}  ,       
&         
         \begin{bmatrix}
           1 \\
           2 \\
           1 \\
           3 \\
           4\\
           \vdots \\
           m\\
           \vdots \\
           m
         \end{bmatrix}  ,       
&         
         \begin{bmatrix}
           1 \\
           1 \\
           2 \\
           3 \\
           4\\
           \vdots \\
           m\\
           \vdots \\
           m
         \end{bmatrix}  ,       
&         
         \begin{bmatrix}
           2 \\
           1 \\
           1 \\
           3 \\
           4\\
           \vdots \\
           m\\
           \vdots \\
           m
         \end{bmatrix} .
         \end{array}
         \]
Applying (\ref{area}) with $a=2$, $r=e=1$ we see that the image of the sequence of entries in the first row under $\varphi$ is the same as as the image of the sequence of entries in the second row under $\varphi$, i.e., $\varphi(2,1,1,2)=\varphi(1,2,1,1)$. Also, letting $a=e=1$ and $r=2$ shows that $\varphi(1,2,1,1)=\varphi(2,1,2,1)$. Since $\varphi$ is a polymorphism, we have that the $k$-tuple
\[(\varphi(2,1,1,2),\varphi(1,2,1,1), \varphi(2,1,2,1),\varphi(3,3,3,3),\dots, \varphi(m,m,m,m),\dots, \varphi(m,m,m,m))\]
is in $R(\mathcal{H})$. However, as we saw above, the first three entries of this tuple are the same and, clearly, so are the last $k-m$ entries. Therefore, this is a $k$-tuple in $R(\mathcal{H})$ with at most $m-1$ distinct entries, contradicting our choice of $m$. 

The argument in the case $m=k$ is similar except that, this time, we consider the following four $k$-tuples:
\[\begin{array}{cccc}\begin{bmatrix}
           1 \\
           3 \\
           2 \\
           4 \\
           5\\
           \vdots \\
           m
         \end{bmatrix}  ,       
&         
         \begin{bmatrix}
           3 \\
           1 \\
           2 \\
           4\\
           5\\
           \vdots \\
           m
         \end{bmatrix}  ,       
&         
         \begin{bmatrix}
           2 \\
           3 \\
           1 \\
           4\\
           5\\
           \vdots \\
           m
         \end{bmatrix}  ,       
&         
         \begin{bmatrix}
           1 \\
           2 \\
           3 \\
           4\\
           5\\
           \vdots \\
           m
         \end{bmatrix} .
         \end{array}
         \]
Note that all four of these tuples are in $R(\mathcal{H})$ because of total symmetry and the fact that $k\geq3$. Applying $\varphi$ to the sequence of entries in the first row yields the same result as applying $\varphi$ to the sequence of entries in the second row by (\ref{area}). Thus, since $\varphi$ is a polymorphism, we again get that $R(\mathcal{H})$ contains an element with at most $m-1$ distinct entries, contradicting our choice of $m$. This completes the proof. 
\end{proof}

We observe now that Lemma \ref{symmetricLem} implies that $H$-colouring is NP-complete for any
non-bipartite graph $H$. As this is the hard part of the the $H$-colouring dichotomy of
\cite{Hcol}, this further attests to the possible independent interest of the lemma. (We must point
out, however, that the proof of Lemma \ref{symmetricLem} relies on results that are stronger than
the $H$-colouring dichotomy.)

Indeed, given a non-bipartite graph $H$ with odd-girth $g \geq 3$, let $d$ be  
smallest odd integer with $g < 3d$. One can show that where $C_{3d}$ is the cycle of girth $3d$
with vertices $x,y$ and $z$ mutually distance $d$ apart, the set
\[ \mathcal{H} = \{ (\phi(x), \phi(y), \phi(z) ) \mid \phi: C_{3d} \to H \} \]
is a non-empty totally symmetric $3$-graph on $V(H)$, containing no constant. So by
Lemma~\ref{symmetricLem}
$\CSP{\mathcal{H}}$ is NP-complete.  For an instance $\mathcal{G}$ of $\CSP{\mathcal{H}}$ one can
construct an instance $G$ of $H$-colouring by replacing any tuple $e = (x_e,y_e,z_e)$ with a copy
of $C_{3d}$ where all vertices are new except for the copies of $x,y$ and $z$, which we take to
be $x_e,y_e$ and $z_e$ respectively.  It is not hard to see that
\[ \mathcal{G} \to \mathcal{H} \iff G \to H, \]
and so that $H$-colouring is NP-complete.

\subsection{The Reduction}

We turn our attention now to the proof of Theorem~\ref{frozenThmGeneral}. We remark that the main construction used in the reduction is partly inspired by a construction used by Fiala and Paulusma~\cite{Fiala} to prove a dichotomy for the problem of deciding whether a graph $G$ admits an $H$-colouring $f$ such that $f\left(N_G(v)\right) = N_H\left(f(v)\right)$ for all $v\in V(G)$. However, rather than using a reduction from the $2$-colouring problem for hypergraphs as was done in~\cite{Fiala}, we will use a reduction from $\CSP{\mathcal{H}}$ for some relational structure $\mathcal{H}$. The following definitions will be useful in the proof.

\begin{defn}
\label{catProdDef}
Given graphs $J_1,\dots,J_t$, let $J_1\times\dots\times J_t$ be the graph with vertex set $V(J_1)\times \dots \times V(J_t)$ where $(u_1,\dots,u_t)$ is adjacent to $(v_1,\dots,v_t)$ if and only if $u_i$ is adjacent to $v_i$ in $J_i$ for $1\leq i\leq t$. We call $\times$ the \emph{categorical product}. 
\end{defn}

\begin{defn}
 Given a graph $J$ and an integer $t$, we let $J^{ t}$ denote the $t$-fold categorical product of $J$ with itself; i.e. $J^{t}:=\underbrace{J\times\dots \times J}_{t}$.
\end{defn}

\begin{defn}
Given a graph $G$ and an integer $t\geq1$ and $1\leq i\leq t$, the function $\varphi_i:V(J^t)\to V(J)$ which maps each vertex to its $i$th coordinate is called the  \emph{projection map} of $J^t$ onto the $i$th coordinate.
\end{defn}

In proving Theorem~\ref{frozenThmGeneral}, we will make use of three simple facts about the categorical product and its projection maps given by the following lemma.

\begin{lem}
\label{prodLem}
Given a graph $J$ and an integer $t\geq1$, the following statements are true: 
\begin{enumerate}[(a)]
\item\label{proj} For $1\leq i\leq t$, the projection map $\varphi_i$ is a homomorphism from $J^t$ to $J$. 
\item\label{iso} If $J$ contains no isolated vertices, then for every vertex $u\in J^t$ and every neighbour $\gamma$ of $\varphi_i(u)$ in $J$ there exists a neighbour $v$ of $u$ in $J^t$ with $\varphi_i(v)=\gamma$. 
\item\label{nonbipProd} If every component of $J$ is non-bipartite, then every component of $J^t$ is non-bipartite. 
\end{enumerate}
\end{lem}

\begin{proof}
If $uv\in E(J^t)$, then, in particular, the $i$th coordinate of $u$ is adjacent to the $i$th coordinate of $v$ in $J$ and so $\varphi_i(u)\varphi_i(v)\in E(J)$. Thus, $\varphi_i$ is a homomorphism and (\ref{proj}) is proved.

Next, we prove (\ref{iso}). Let $u=(\alpha_1,\dots,\alpha_t)$ be a vertex of $V(J^t)$. Since none of $\alpha_1,\dots,\alpha_t$ are isolated in $J$, we get that $u$ is adjacent to $v$ where $v = (\gamma_1,\dots,\gamma_t)$ and $\gamma_j$ is a neighbour of $\alpha_j$ for $1\leq j\leq t$. Since $\gamma_i$ can be chosen to be an arbitrary neighbour of $\alpha_i$, we see that (\ref{iso}) is proved. 

The contrapositive of (\ref{nonbipProd}) is now immediate by \eqref{proj}.
Indeed, if $J^t$ contains an odd cycle, they it must map to an odd cycle in $J$ via any of the
projections.
\end{proof}

Finally, we present the proof of Theorem~\ref{frozenThmGeneral}.

\begin{proof}[Proof of Theorem~\ref{frozenThmGeneral}]
The polynomial side was already proved earlier in this section. So, we prove only the NP-complete side. Let $H$ be a graph containing no component consisting of a single reflexive vertex such that either 
\begin{itemize}
\item $H$ has a component isomorphic to $K_2$ and the freezer of $H$ is non-bipartite, or
\item $H$ does not have a component isomorphic to $K_2$ and $H$ has at least one non-thermal component with at least three vertices. 
\end{itemize}
Our goal is to show that \Hfrz{H} is NP-complete. Let $H'$ be a subgraph of $H$ such that
\begin{itemize}
\item If $H$ has a component isomorphic to $K_2$, then $H'$ is obtained from $H$ by deleting all thermal components and all components $F$ such that $F[S_F]$ is bipartite.
\item If $H$ does not have a component isomorphic to $K_2$, then $H'$ is obtained from $H$ by deleting all isolated vertices and all thermal components. 
\end{itemize}

Note that every component of $H'$ has at least three vertices. Let $\mathcal{D}$ be the collection of all $D\subseteq V(H')$ such that $D$ is distinguishing for some vertex $\alpha$ in some component $F$ of $H'$. We construct a \kgraph~ $\mathcal{H}$ with vertex set $V(\mathcal{H})=V(H')$ as follows:
\begin{itemize}
\item Define $k:=\max\left(\{3\}\cup \{|D|: D\in\mathcal{D}\}\right)$. 
\item Let $R(\mathcal{H})$ consist of all of the $k$-tuples $(r_1,\dots,r_k)$ such that $\{r_1,\dots,r_k\}\in \mathcal{D}$. 
\end{itemize}
Note that $k \geq 3$ and $\mathcal{H}$ is totally symmetric by construction; also $\mathcal{H}$ contains no constant tuple by Lemma~\ref{notK1K2} and the fact that every component of $H'$ has at least three vertices. Therefore, by Lemma~\ref{symmetricLem}, we have that \CSP{\mathcal{H}} is NP-complete. Our goal is to reduce \CSP{\mathcal{H}} to \Hfrz{H}. 

To this end, let $\mathcal{G}$ be any instance of \CSP{\mathcal{H}}. For each $r=(r_1,\dots,r_k)\in R(\mathcal{G})$, we create a vertex $w_r$ corresponding to $r$ and let $W:=\left\{w_r: r\in R(\mathcal{G})\right\}$. We first construct the incidence graph of $\mathcal{G}$: a bipartite graph $B$ with bipartition $\left(V(\mathcal{G}),W\right)$ where a vertex $v\in V(\mathcal{G})$ is adjacent to $w_r$ with $r=(r_1,\dots,r_k)$ if $v\in \{r_1,\dots,r_k\}$. 

Let $J$ be the freezer of $H'$, let $\tilde{J}:=J^{|V(J)|}$ and let $\tilde{z}$ be a vertex of $\tilde{J}$ with all of its coordinates distinct. We construct a graph $G$ from $B$ by adding a copy $\tilde{J}_v$ of $\tilde{J}$ for each $v\in V(\mathcal{G})$ and identifying $v$ with the copy $\tilde{z}_v$ of $\tilde{z}$ in $\tilde{J}_v$. Note that $|V(G)|=O\left(|V(\mathcal{G})| + \left|R(\mathcal{G})\right|\right)$ and that $G$ can be constructed from $\mathcal{G}$ in polynomial time. 

We claim that $\mathcal{G}$ admits a homomorphism to $\mathcal{H}$ if and only if $G$ has a frozen $H$-colouring. First, we suppose that $\mathcal{G}$ admits a homomorphism $f$ to $\mathcal{H}$ and construct a frozen $H$-colouring $g$ of $G$. First, let $g(v)=f(v)$ for all $v\in V(\mathcal{G})$. Next, for every $r = (r_1,\dots, r_k)\in R(\mathcal{G})$, we have that $\{f(r_1),\dots,f(r_k)\}\in \mathcal{D}$. So, we can let $\alpha_r$ be the unique vertex of $H$ such that $\{f(r_1),\dots,f(r_k)\}\subseteq N_H(\alpha_r)$ and let $g(w_r)=\alpha_r$. Now, for each $v\in V(\mathcal{G})$, we colour the vertices of $\tilde{J}_v$ with the projection map onto the coordinate of $\tilde{z}_v$ corresponding to the colour $g(v)$ (note that this is clearly consistent with the colour already chosen for $v$). We note here that by Lemma~\ref{prodLem}(\ref{iso}) the projection map is locally surjective.  Using Lemmas~\ref{distinguishingLem} and~\ref{nbhdDistinguishes}, together with the fact that $J$ is the freezer of $H'$, shows this projection is frozen.

We claim that the neighbourhood of every vertex of $G$ is mapped to a set in $\mathcal{D}$ which will imply that $g$ is frozen by Lemma~\ref{distinguishingLem}. As we have already seen, the neighbourhood of each $w_r\in W$ maps (by $f$ and so by $g$) onto a set of $\mathcal{D}$. Given $u\in V(\tilde{J}_v)$ for some $v\in V(\mathcal{G})$ (including the case $u=v$), let $F$ be the component of $H'$ containing $g(u)$. By Corollary~\ref{frozenGoesToS}, we have that $g(u)\in S_F$. By definition, the graph $J$ contains the graph $F[S_F]$ as a component. So, applying Lemma~\ref{prodLem} (\ref{iso}) and the fact that $J$ has no isolated vertices (by definition of $H'$), we see that, for every $\beta\in N_F(g(u))\cap S_F$, there is a neighbour of $u$ of  colour $\beta$. Thus,  $g(N_G(u))\in \mathcal{D}$ by Lemma~\ref{nbhdDistinguishes}. Therefore, $g$ is frozen. 

Now, for the other direction, suppose that there exists a frozen $H$-colouring $g$ of $G$. We claim that the restriction of $g$ to $V(\mathcal{G})$ is a homomorphism from $\mathcal{G}$ to $\mathcal{H}$. Let $r=(r_1,\dots,r_k)\in R(\mathcal{G})$ be arbitrary; our goal is to show that $g(\{r_1,\dots,r_k\})\in \mathcal{D}$. Let $F$ be the component of $H$ containing $g(w_r)$. Then $F$ must be non-thermal and, by Lemma~\ref{distinguishingLem}, we must have that $g(\{r_1,\dots,r_k\})$ is distinguishing for $g(w_r)$. If $F$ is a component of $H'$, then we have $g(\{r_1,\dots,r_k\})\in\mathcal{D}$ and we are done. So, we assume that $F$ is a non-thermal component of $H$ which is not a component of $H'$. Clearly $F$ cannot have only one vertex as $H$ does not contain a component consisting of a single reflexive vertex and $w_r$ is not an isolated vertex of $G$. Therefore, the only possibility is that $H$ contains a component isomorphic to $K_2$ and $F[S_F]$ is bipartite. For an arbitrary neighbour $v$ of $w_r$, we have that the component of $\tilde{J}_v$ containing $\tilde{z}_v$ is mapped by $g$ into $S_F$ by Corollary~\ref{frozenGoesToS}. However this is impossible since  $F[S_F]$ is bipartite and no component of $\tilde{J}$ is bipartite. Indeed, by construction of $H'$ (recall that we are in the case that $H$ contains a component isomorphic to $K_2$), we have that every component of $J$ is non-bipartite  and so every component of $\tilde{J}$  is non-bipartite by Lemma~\ref{prodLem} (\ref{nonbipProd}). This completes the proof.
\end{proof}

\section{Concluding Remarks}
\label{conclusion}

The complexity of the \Hrec{H} problems are still wide open for general $H$. For \Hrec{H}, it is very tempting to conjecture that the problem is always either PSPACE-complete or solvable in polynomial time (indeed, there is no known example which would refute such a conjecture). 

Related to the \Hrec{H} problem of finding paths between maps in $\bHom(G,H)$, is the problem of
determing the diameter of components of $\bHom(G,H)$.
Given a fixed graph $H$ if, for every graph $G$,  every component of $\bHom(G,H)$ has diameter bounded by a polynomial in $|V(G)|$ (where the polynomial depends only on $H$), then the \Hrec{H} problem is in NP as every ``yes'' instance has a polynomial sized certificate. This raises the following natural question. 

\begin{ques}
\label{diamQ}
For which graphs $H$ does there exist an integer $k=k(H)$ such that, for every graph $G$, every component of $\bHom(G,H)$ has diameter at most $|V(G)|^k$?
\end{ques}

It seems plausible that, if \Hrec{H} is in NP (in particular, it if is in P), then the components of $\bHom(G,H)$ ``should'' have polynomial diameter, since a reconfiguration sequence is the most natural certificate. Therefore, answering Question~\ref{diamQ} could be an important step towards clarifying the complexity of the \Hrec{H} problem. We should mention that Bonsma and Cereceda~\cite{Bonsma} proved that, for $k\geq4$, there exists graphs $G$ such that $\bHom(G,K_k)$ contains a component of superpolynomial diameter. We should also mention that the complexity of the problem of finding the shortest path between two vertices of $\bHom(G,H)$ has been considered (in the case that $H$ is a clique) by Johnson, Kratsch, Kratsch, Patel and Paulusma~\cite{shortest}. 


\begin{ack}
The authors would like to thank Sean McGuinness for participating in many enjoyable discussions on topics related to those covered in this paper. 
\end{ack}

\bibliographystyle{plain}
\bibliography{reconfig}

\end{document}